\documentclass[10pt]{amsart}
\usepackage{amsmath,amssymb,amsthm,array,hyperref,color,cite,graphicx,multirow,pdflscape,mathrsfs,float}
\usepackage[all,arc]{xypic}
\usepackage[all]{xypic}

\newtheorem{thm}{Theorem}[section]
\newtheorem{cor}[thm]{Corollary}
\newtheorem{lem}[thm]{Lemma}
\newtheorem{prop}[thm]{Proposition}
\theoremstyle{definition}

\newtheorem{df}[thm]{Definition}

\newtheorem{rmk}[thm]{Remark}
\newtheorem{qn}[thm]{Question}

\newcommand{\R}{\mathbb R}

\newcommand{\Z}{\mathbb Z}
\newcommand{\Q}{\mathbb Q}
\newcommand{\Sph}{\mathbb S}

\newcommand{\Hom}{\textup{Hom}}

\newcommand{\Ext}{\textup{Ext}}
\newcommand{\id}{\textup{id}}

\newcommand{\holim}{\textup{holim}\,}

\newcommand{\ra}{\longrightarrow}

\newcommand{\sma}{\wedge}

\newcommand{\ti}{\widetilde}
\newcommand{\simar}{\overset\sim\longrightarrow}

\newcommand{\mc}{\mathcal}

\newcommand{\sk}{\textup{Sk}}

\newcommand{\cyc}{\textup{cyc}}
\newcommand{\tr}{\textup{tr}\,}

\definecolor{grey}{gray}{0.4}

\makeatletter
\ifx\SK@label\undefined\let\SK@label\label\fi
 \let\your@thm\@thm
 \def\@thm#1#2#3{\gdef\currthmtype{#3}\your@thm{#1}{#2}{#3}}
 \def\mylabel#1{{\let\your@currentlabel\@currentlabel\def\@currentlabel
  {\currthmtype~\your@currentlabel}
 \SK@label{#1@}}\label{#1}}
 
\makeatother

\title{The topological cyclic homology of the dual circle}
\author{Cary Malkiewich}

\begin{document}







\maketitle

\begin{abstract}
We give a new proof of a result of Lazarev, that the dual of the circle $S^1_+$ in the category of spectra is equivalent to a strictly square-zero extension as an associative ring spectrum. As an application, we calculate the topological cyclic homology of $DS^1$ and rule out a Koszul-dual reformulation of the Novikov conjecture.
\end{abstract}

\parskip 0ex
\tableofcontents
\parskip 2ex





%
%



\section{Introduction.}


It is an elementary theorem in topology that a reduced suspension $\Sigma X$ has vanishing cup products in positive degrees. In other words, the cohomology ring $H^*(\Sigma X)$ is a square-zero extension of $\Z$.

One might ask whether this theorem lifts to the commutative ring spectrum
\[ D(\Sigma X) = F((\Sigma X)_+,\Sph^0) \]
with multiplication given by the diagonal map of $\Sigma X$. In other words, is it equivalent as a ring spectrum to a square-zero extension of the sphere spectrum,
\[ \Sph \vee \tilde{D}(\Sigma X) = \Sph \vee F(\Sigma X,\Sph^0)? \]
The answer is as follows.
\begin{thm}\label{thm:intro_main}
There is an equivalence in the homotopy category of augmented $A_\infty$ ring spectra
\[ D(\Sigma X) \simeq \Sph \vee \tilde{D}(\Sigma X) \]
\end{thm}

This theorem, without the prefix ``augmented,'' was first established by Lazarev, using obstruction theory for maps of $A_\infty$ ring spectra (\cite{lazarev2004spaces}, 4.1). In this paper, we give a new proof that is more constructive. The idea is that the diagonal map $S^1 \ra S^2$ is nullhomotopic, and we extend this nullhomotopy to the structure of an $A_\infty$ coalgebra on $S^1$. Our construction is fairly geometric, so it provides useful intuition for Theorem \ref{thm:intro_main} as well.

We also remark that Theorem \ref{thm:intro_main} cannot be lifted to an equivalence of augmented $E_\infty$ ring spectra, when $X$ is a finite complex with nontrivial cohomology. Indeed, it is an old observation of Miller and McClure that the Steenrod operations on $\alpha \in H^*(\Sigma X)$ arise from the extended power operations on the ring spectrum $D(\Sigma X)$ \cite[III.1.2]{h_infinity}. Restricting to reduced cohomology $\alpha \in \ti H^*(\Sigma X)$, this leads to a direct construction of $\textup{Sq}^i(\alpha)$ from the map of spectra
\begin{equation}\label{eq:steenrod}
\xymatrix @R=.8em{
(E\Sigma_2)_+ \sma_{\Sigma_2} \ti{D}(\Sigma X)^{\sma 2} \ar[r] & \ti{D}(\Sigma X)
}
\end{equation}
given by the dual of the diagonal of $\Sigma X$. The $E_\infty$ variant of Theorem \ref{thm:intro_main} would imply that this map of spectra vanishes, and hence all the Steenrod operations on $\ti H^*(\Sigma X)$ would vanish (including $\textup{Sq}^0$).

The motivation for Theorem \ref{thm:intro_main} comes from the algebraic $K$-theory of ring spectra. The ring spectrum $DX$ is Koszul dual to $\Sigma^\infty_+ \Omega X$ for finite, simply-connected $X$ \cite{blumberg2011derived}. This begs the question of whether the behavior of $K(DX)$ as $X$ varies has some parallel with the behavior of Waldhausen's functor $A(X)$ \cite{1126}.

A reasonable guess is given by the following analogue of the Novikov conjecture. Since $K(DX)$ is a contravariant, homotopy-invariant functor of the space $X$, it has a homotopy sheafification map (also known as a coassembly map or a Thomason limit problem map)
\begin{equation}\label{eq:intro_coassembly}
K(DX) \overset{c\alpha}\ra F(X_+,K(\Sph)).
\end{equation}
\begin{qn}\label{q:dual_novikov}
Is the map \eqref{eq:intro_coassembly} surjective on the rational homotopy groups $\pi^\Q_{\geq 0}$, for finite complexes $X$?
\end{qn}

By Theorem \ref{thm:intro_main}, this question is easiest to approach when $X$ is itself a reduced suspension. We recall that the algebraic $K$-theory of a ring spectrum $R$ admits a trace map to the topological Hochschild homology $THH(R)$ \cite{bhm,madsen_survey,dundas2012local}, which has a useful splitting when $R$ is a square-zero extension (Prop \ref{prop:THH_of_square_zero_extension}, cf. \cite[V.3.2]{dundas2012local}, \cite[(6.2.1)]{hesselholt1997k}). In the case where $R = D(\Sigma X)$ this implies
\begin{cor}\label{cor:intro_thh_splitting}
There is a natural equivalence of genuine $S^1$-spectra
\[ THH(D\Sigma X) \simeq \Sph \vee \Sigma^{-1}\left(\bigvee_{n=1}^\infty \ti D(X^{\sma n}) \sma_{C_n} S^1_+ \right) \]
where as above $\ti DX = F(X,\Sph)$. In particular,
\[ THH(DS^1) \simeq \Sph \vee \bigvee_{n=1}^\infty \Sigma^{-1}_+ (S^1/C_n) \]
where $C_n \leq S^1$ denotes the cyclic group of order $n$.
\end{cor}

This splitting interacts nicely with the ``cyclotomic'' structure on $THH$ defined in \cite{bhm}. We may therefore use it to compute the topological cyclic homology $TC(R)$, a finer invariant which also receives a trace map from $K(R)$. The calculation is particularly tractable for the ring spectrum $DS^1$.
\begin{thm}\label{thm:intro_tc_splitting}
There is an equivalence after $p$-completion
\[ TC(D S^1) \simeq \Sph \vee \Sigma \mathbb{CP}^\infty_{-1} \vee \bigvee_{n \in \mathbb{N}} E \]
in which $E$ is the homotopy fiber of the wedge of transfers
\[ \left(\bigvee_{k=0}^\infty \Sigma^\infty_+ BC_{p^k}\right) \ra \Sigma^{-1}\Sigma^\infty_+ B\Z \]
along $S^1$-bundles of the form
\[ B\Z \overset{\cdot p^k}\ra B\Z \ra BC_{p^k} \]
\end{thm}
\begin{cor}\label{cor:intro_coassembly_zero}
The coassembly map for $K(DS^1)$ is not surjective on rational homotopy groups in degree 4. More generally, it is rationally zero in degree $4i$ when $i \geq 1$ and there exists a regular prime $p$ such that $p \geq 2i+3$.
\end{cor}
This demonstrates that $K(DX)$ and $A(X)$ have markedly different behavior as $X$ varies. It gives a negative answer to Question \ref{q:dual_novikov}, although an affirmative answer may still be possible in the simply-connected case. We take this result as an indication that $TC(DX)$ will be a useful source of further clues as to the behavior of $K(DX)$.

We emphasize that the calculations in the second half of this paper require the full force of Theorem \ref{thm:intro_main}. Without an equivalence of $A_\infty$ rings, one can deduce the splitting in Corollary \ref{cor:intro_thh_splitting} nonequivariantly, but this is not powerful enough to draw conclusions about the topological cyclic homology.

The paper is organized as follows. In section \ref{sec:main_proof} we prove Theorem \ref{thm:intro_main}. In section \ref{sec:thh_splitting} we give the splitting of $THH(D\Sigma X)$, proving Corollary \ref{cor:intro_thh_splitting}. In section \ref{sec:review} we review the needed concepts in equivariant stable homotopy theory, and then calculate $TC(DS^1)$ in section \ref{sec:tc}, proving Theorem \ref{thm:intro_tc_splitting} and Corollary \ref{cor:intro_coassembly_zero}.

The author is pleased to acknowledge Andrew Blumberg, Ralph Cohen, and Randy McCarthy for several helpful conversations throughout this project, and Nick Kuhn for insightful comments on an earlier version of the paper. The author would also like to thank an anonymous referee for comments that substantially improved the exposition in the paper. This paper represents a part of the author's Ph.D. thesis, written under the direction of Ralph Cohen at Stanford University.

\section{Proof of Theorem \ref{thm:intro_main}.}\label{sec:main_proof}

We will work in the category of orthogonal spectra from \cite{mmss}. Let $X$ be a based CW complex, and recall our definitions
\[ D(\Sigma X) = F((\Sigma X)_+,\Sph), \qquad \ti D(\Sigma X) = F(\Sigma X,\Sph), \]
where $\Sph$ denotes the sphere spectrum. The obvious collapse maps from $(\Sigma X)_+$ into $S^0$ and $\Sigma X$ give an equivalence of spectra
\[ \Sph \vee \ti D(\Sigma X) \simar D(\Sigma X). \]
This becomes an equivalence of ring spectra when we endow the term $\ti D(\Sigma X) = F(\Sigma X,\Sph)$ with a multiplication by the dual of the diagonal map
\begin{equation}\label{eq:diagonal}
\Sigma X \ra \Sigma X \sma \Sigma X
\end{equation}
Our goal is to show that $\ti D(\Sigma X)$, with this multiplication, is equivalent as a non-unital $A_\infty$ ring spectrum to $\ti D(\Sigma X)$ with the zero multiplication.

It is elementary that the diagonal map \eqref{eq:diagonal} is nullhomotopic as a map of based spaces. For instance, we may model $S^n$ by $\R^n$ modulo the complement of the open cube $(0,1)^n$, and take the nullhomotopy which at time $t \in [0,1]$ is given by the formula
\[ \begin{array}{c}
S^1 \sma X \ra S^1 \sma S^1 \sma X \sma X \\
(s,x) \mapsto (s + t, s, x, x)
\end{array}
\]
\vspace{-1em}

\begin{figure}[H]
\centerline{\hspace{5em}\includegraphics[scale=1.3]{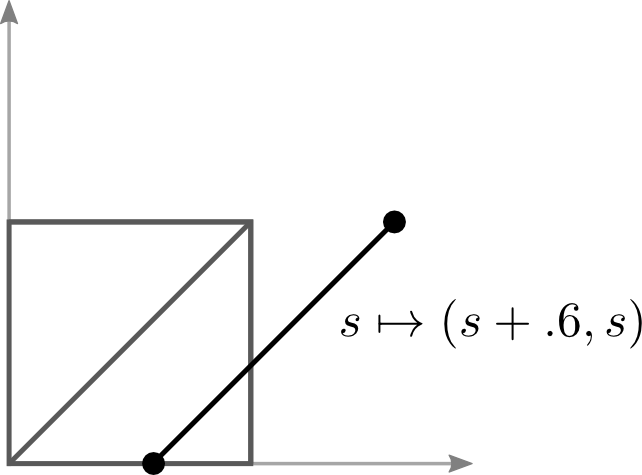}}
\caption{A nullhomotopy of the diagonal of the circle.}
\end{figure}\vspace{-2em}

In order to get an equivalence of ring spectra, it is necessary to add more structure, incorporating all compositions of these intermediate diagonal maps into a larger contractible space. To make this precise we use the theory of operads \cite{may_geom_iterated}. Consider a (non-symmetric) operad $\mc O$ whose $n$th space is
\[ \mc O(n) = (\R_{\geq 0})^{n-1} = \{ (t_1,\ldots,t_{n-1}) : t_i \geq 0 \} \]
To define the operad structure on $\mc O$ we think of $\mc O(n)$ as the subspace of $(\R_{\geq 0})^n$ whose last coordinate is zero.
Then the composition comes from adding the $t_i$ together:
\[ \mc O(k) \times \mc O(j_1) \times \ldots \times \mc O(j_k) \ra \mc O(j_1 + \ldots + j_k) \]
\[ (t_1,\ldots,t_k) \circ ((s^1_1,\ldots,s^1_{j_1}),\ldots,(s^k_1,\ldots,s^k_{j_k})) \mapsto (t_1 + s^1_1,\ldots,t_1 + s^1_{j_1},\ldots, t_k + s^k_1, \ldots t_k + s^k_{j_k}) \]
Clearly if $t_k = 0$ and $s^k_{j_k} = 0$ then $t_k + s^k_{j_k} = 0$, so our desired subspace is preserved under these maps.
It is straightforward to check that these maps are associative in the correct sense, so that $\mc O$ is an operad.

\begin{rmk}
$\mc O$ is almost an $A_\infty$ operad, except that $\mc O(0)  = \emptyset$ instead of being contractible. Therefore algebras $R$ over $\mc O$ are spectra with an $A_\infty$ multiplication but no unit map $\Sph \to R$.
\end{rmk}

Now we define the action of $\mc O$ on $\ti D(\Sigma X)$. To handle the case of $X$ infinite, we allow ourselves to take a fibrant replacement $f\Sph$ of $\Sph$ as an associative ring spectrum, using the model structure of \cite[Thm 12.1(iv)]{mmss}. Now we let the point $(t_1,\ldots,t_{n-1})$ give the map
\begin{equation}\label{eq:o_action}
F(\Sigma X,f\Sph)^{\sma n} \ra F((\Sigma X)^{\sma n},(f\Sph)^{\sma n}) \ra F(\Sigma X,f\Sph)
\end{equation}
where the first leg is adjoint to a smash product of $n$ evaluation maps
\[ (\Sigma X)^{\sma n} \sma F(\Sigma X,f\Sph)^{\sma n} \ra (f\Sph)^{\sma n} \]
and the second leg is induced by the multiplication on the fibrant sphere $f\Sph$ and by the map of spaces
\[ \begin{array}{c}
S^1 \sma X \ra S^n \sma X^{\sma n} \\
(s,x) \mapsto (s + t_1, s + t_2, \ldots, s + t_{n-1}, s, x, x, \ldots, x)
\end{array}
\]
\vspace{-1em}

\begin{figure}[H]
\centering{\mbox{\hspace{5em}\includegraphics[scale=1.3]{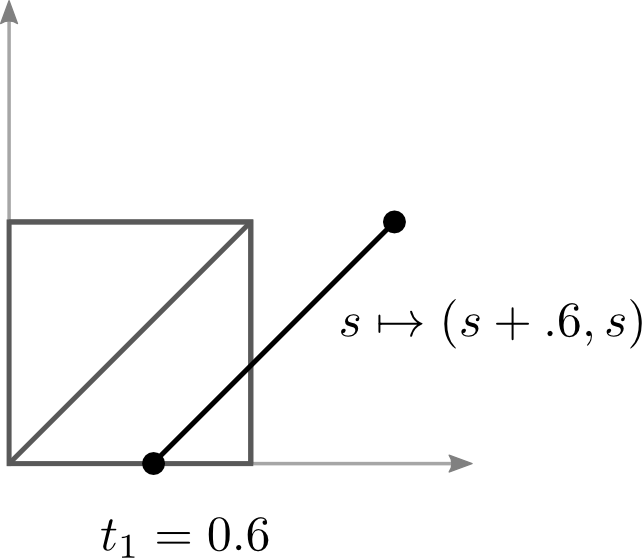}\includegraphics[scale=1.3]{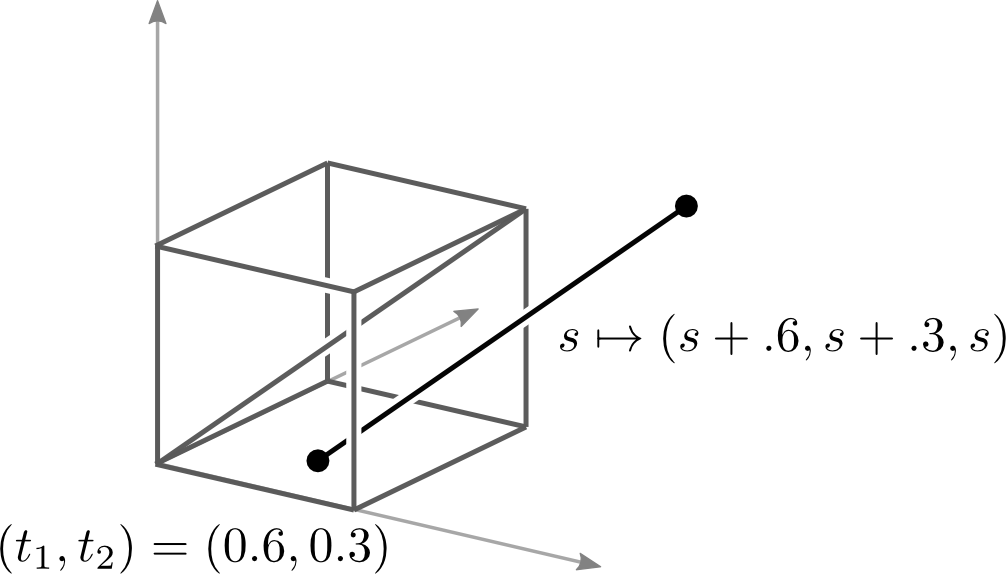}}}
\caption{The action of $\mc O(2)$ and $\mc O(3)$ on the dual of a suspension.}
\end{figure}\vspace{-2em}

As above, we are modeling $S^n$ by $\R^n$ modulo the complement of the open set $(0,1)^n$. It is straightforward to check that this respects the appropriate compositions, thereby making $\ti D(\Sigma X)$ into an $\mc O$-algebra. In fact, it is more natural to say that $\Sigma X$ is an ``$\mc O$-coalgebra,'' since our collection of maps
\[ \mc O(n)_+ \sma \Sigma X \ra (\Sigma X)^{\sma n} \]
are compatible with composition in a manner dual to that found in the definition of an $\mc O$-algebra.

For the penultimate step, we check that the following rules define two suboperads of $\mc O$:
\[ \begin{array}{ccl} \mc A(n) &=& \{(0,\ldots,0)\} \qquad \forall n \geq 1 \\
\\
\mc O'(1) &=& \{(0)\} \\
\mc O'(n) &=& \{(t_1,\ldots,t_{n-1}) : t_i \geq 1 \ \forall i \} \qquad \forall n \geq 2
\end{array} \]
Of course, $\mc A$ is the (non-unital) associative operad, acting on $F(\Sigma X,\Sph)$ by the usual multiplication given by the diagonal map.
On the other hand, $\mc O'(n)$ acts only by the zero multiplication when $n \geq 2$, sending everything to the basepoint.
So if we let $\mc Z$ be another copy of the non-unital associative operad, with $\mc Z(n)$ acting on $F(\Sigma X,\Sph)$ by the zero multiplication for $n \geq 2$, we get a commuting diagram of operads
\begin{equation}\label{eq:operads}
\xymatrix @C=-1em{
& \mbox{\hspace{4.5em}} && \mc End(\ti D(\Sigma X)) && \mbox{\hspace{5em}} & \\
\mc A \ar[rr]^-\sim \ar[rrru] \ar[rrrd]_-\sim & & \mc O \ar[ru] \ar[rd]^-\sim & & \mc O' \ar[rr]^-\sim \ar[ll]_-\sim \ar[lu] \ar[ld]_-\sim & & \mc Z \ar[lllu] \ar[llld]^-\sim \\
&&& \mc A &&& }
\end{equation}
This ends the intuitive phase of the proof; it remains to recall how such equivalences of operads lead to equivalences of algebras. Each operad $\mc C$ has an associated monad $C$, an endofunctor of spectra of the form
\[ C(Y) = \bigvee_{n \geq 0} \mc C(n)_+ \sma Y^{\sma n} \]
when $\mc C$ is a non-symmetric operad. The functor $C$ has a monoid structure in the form of natural transformations $\id \to C$ and $C \circ C \to C$, and this allows one to form a bar construction $B(D,C,Y)$ for any $\mc C$-algebra $Y$ and map of monads $C \to D$. The diagram of operads \eqref{eq:operads} then gives a zig-zag
\begin{equation}\label{eq:bar_constructions}
\xymatrix{
\ti D(\Sigma X) & B(\mc A,\mc A,\ti D(\Sigma X)) \ar[l]_-\sim \ar[r]^-\sim & B(\mc A,\mc O,\ti D(\Sigma X)) \\
\ti D(\Sigma X)_{\textup{zero mult}} & B(\mc A,\mc Z,\ti D(\Sigma X)) \ar[l]_-\sim & B(\mc A,\mc O',\ti D(\Sigma X)) \ar[u]_-\sim \ar[l]_-\sim }
\end{equation}
The maps are all equivalences of $\mc A$-algebras (i.e. non-unital associative rings) by standard properties of the bar construction \cite[\S 9]{may_geom_iterated}, so long as we make a few careful cofibrant replacements. There are a few ways to do this, and we describe just one possible method. We first observe that associative ring spectra form a cofibrantly generated model category, and so once we have a zig-zag of simplicial associative rings, we can always replace it by a level-equivalent zig-zag of Reedy cofibrant simplicial associative rings.
It remains to check that our maps of simplicial objects are equivalences on each simplicial level. Expanding out the definition of $A(C(\ldots(C(Y))\ldots))$ for any of our monads $C$, we get a wedge sum in which each summand is a smash product of a CW complex and a spectrum of the form $F(\Sigma X,f\Sph)^{\sma n}$. We may replace each such term with $F((\Sigma X)^{\sma n},f\Sph)$, and the relevant operad structure still exists, because it was defined through a coalgebra structure on $\Sigma X$. 
This replacement gives a similar zig-zag of simplicial ring spectra, whose levels now have the homotopy type that we expect.

Once we are assured that the maps of \eqref{eq:bar_constructions} are equivalences of non-unital associative rings, we apply $\Sph \vee -$ to get equivalences of augmented associative rings. This finishes our proof of Theorem \ref{thm:intro_main}.

\section{Topological Hochschild homology of $D(\Sigma X)$.}\label{sec:thh_splitting}

We briefly recall that the topological Hochschild homology of an ring spectrum $R$ is the geometric realization of the simplicial spectrum whose $(n-1)$st level is $R^{\sma n}$, the smash product of $n$ copies of $R$. In fact, this produces not just a simplicial object but a cyclic object in spectra (cf. \cite{dhk,jones1987cyclic,bhm,madsen_survey}). The realization is a genuine $S^1$-spectrum with ``cyclotomic'' structure, meaning there are compatible equivalences of geometric fixed point spectra
\begin{equation}\label{eq:cyclotomic}
c_m: \Phi^{C_m} THH(R) \simeq THH(R)
\end{equation}
for every finite subgroup $C_m \cong \Z/m\Z$ of the circle group $S^1$ \cite{bokstedt1985topological}. The choice of model does not matter for our purposes, but for definiteness we use the norm model for $THH$ \cite{angeltveit2014relative,part1}, which produces an orthogonal $S^1$-spectrum \cite{mandell2002equivariant}. We tacitly assume that $R$ is made cofibrant as an orthogonal spectrum when we apply $THH$; this is possible by \cite[12.1.(iv)-(v)]{mmss}. Then all the aforementioned structure on $THH$ is an invariant of $R$ as an $A_\infty$ ring spectrum \cite{angeltveit2014relative,part1}.

Now suppose that $R$ is a square-zero extension of $\Sph$ by a cofibrant orthogonal spectrum $M$. We indicate a simple proof of the splitting of $THH(\Sph \vee M)$ in the norm model, analogous to the splitting that occurs in the B\" okstedt model \cite[V.3.2]{dundas2012local}.
\begin{prop}\label{prop:THH_of_square_zero_extension}
There is an isomorphism of genuine orthogonal $S^1$-spectra
\[ THH(\Sph \vee M) \simeq \Sph \vee \bigvee_{n=1}^\infty (S^1_+ \sma S^{\overline{\rho_{C_n}}}) \sma_{C_n} M^{\sma n} \]
where $\overline{\rho_{C_n}}$ denotes the reduced regular representation.
\end{prop}

\begin{proof}
The relevant cyclic spectrum is at level $k-1$
\[ (\Sph \vee M)^{\sma k} \cong \bigvee_{S \subset \{1,\ldots,k\}} M^{\sma |S|} \]
This splits as a cyclic object, with one summand for each $n \geq 0$, consisting of all the terms with $|S| = n$. When $n = 0$ we get the constant simplicial object on $\Sph$ and so the realization is $\Sph$. For each $n > 0$, we analyze the resulting cyclic spectrum $X_\bullet$ using the language of \cite[\S 2.2]{part1}. The cyclic latching maps are all isomorphisms above level $n-1$, so the realization is attained at the cyclic $(n-1)$-skeleton. Since $X_k = *$ when $k < n-1$, the cyclic latching square simplifies to the pushout square
\[ \xymatrix{
X_{n-1} \sma_{C_n} \partial\Lambda^{n-1}_+ \ar[r] \ar[d] & X_{n-1} \sma_{C_n} \Lambda^{n-1}_+ \ar[d] \\
{*} \ar[r] & |\sk_{n-1}^\cyc X_\bullet| \cong |X_\bullet| } \]
where $\Lambda^{n-1} \cong \Delta^{n-1} \times S^1$ is the standard cyclic $n$-simplex. Since $X_{n-1} = M^{\sma n}$ we conclude the realization is
\[ (\Lambda^{n-1}/\partial \Lambda^{n-1}) \sma_{C_n} M^{\sma n}. \]
The action of the cyclic group $C_n$ on $\Lambda^{n-1}$ is by the cycle map in the category $\mathbf \Lambda$. There is a choice of homeomorphism $\Lambda^{n-1} \cong \Delta^{n-1} \times S^1$ under which the $C_n$-action rotates the vertices of $\Delta^{n-1}$ and shifts the circle by $\frac1{n}$ of a full rotation \cite[p.200]{madsen_survey}, and this gives the proposition.

\end{proof}

\begin{rmk}
Although we will not use this fact, we can observe that the simplicial filtration of $(\Lambda^{n-1}/\partial \Lambda^{n-1}) \sma_{C_n} M^{\sma n}$ is only nontrivial at levels $(n-1)$ and $n$. Therefore the B{\"o}kstedt spectral sequence for $H_*(THH(\Sph \vee M))$ with field coefficients always collapses at the $E^2$ page.
\end{rmk}

\begin{cor}\label{cor:thh_splitting}
There is an equivalence of genuine $S^1$-spectra
\[ THH(D \Sigma X) \simeq \Sph \vee \Sigma^{-1}\left(\bigvee_{n=1}^\infty S^1_+ \sma_{C_n} (c\ti DX)^{\sma n} \right) \]
where the $c$ denotes cofibrant replacement as an orthogonal spectrum.
\end{cor}

\begin{proof}
We replace $D(\Sigma X)$ by the cofibrant square-zero extension $\Sph \vee c\ti{D}(\Sigma X)$. Since they are equivalent as $A_\infty$ rings, this does not change the topological Hochschild homology up to equivalence. The above splitting rearranges to
\[ THH(\Sph \vee c\ti D(\Sigma X)) \cong \Sph \vee \bigvee_{n=1}^\infty S^1_+ \sma_{C_n} (S^{\overline\rho_{C_n}} \sma (c\ti D\Sigma X)^{\sma n}) \]
We arrange our cofibrant models of $\ti DX$ and $\ti D\Sigma X$ to be equipped with an equivalence
\[ \xymatrix{ F_1 S^0 \sma c\ti DX \ar[r]^-{e}_-\sim & c\ti D\Sigma X } \]
where $F_1 S^0$ is the model for $\Sigma^{-1}\Sph$ from \cite[1.3]{mmss}. The $n$-fold smash power of $e$ gives maps of orthogonal $C_n$-spectra
\[ \xymatrix{ F_1 S^0 \sma (c\ti DX)^{\sma n} &
S^{\overline\rho_{C_n}} \sma F_{\rho_{C_n}} S^0 \sma c(\ti DX)^{\sma n} \ar[l]_-\sim \ar[r]^-{\id \sma e^{\sma n}}_-\sim &
S^{\overline\rho_{C_n}} \sma (c\ti D\Sigma X)^{\sma n} } \]
These maps are equivalences on all the geometric fixed points \cite[B.96]{hhr} and so they are genuine $C_n$-equivalences. The spectra are cofibrant in the model category of $C_n$-spectra \cite[\S III.4]{mandell2002equivariant} and so the left Quillen functor $S^1_+ \sma_{C_n} -$ preserves these equivalences.
\end{proof}

\begin{rmk}\label{rem:cm}
In the norm model of $THH$, the equivalences $c_m$ in \eqref{eq:cyclotomic} come from homeomorphisms of orthogonal spectra \cite{angeltveit2014relative,part1}. This allows us to easily identify $c_m$ for $R = \Sph \vee M$ as the map taking the $C_m$-geometric fixed points of the $(mn)$th summand to the $n$th summand by an equivalence. The same conclusion can be drawn from Bokstedt's model as well, although one has to work a little more.
\end{rmk}

\section{Review of equivariant stable homotopy and transfers.}\label{sec:review}

Our final task is to compute the topological cyclic homology of the dual circle, and in this section we recall the necessary preliminaries. We first recall that for each covering space $f: E \to B$, one may use a Pontryagin-Thom collapse to define a transfer map 
\[ \tr_f: \Sigma^\infty_+ B \to \Sigma^\infty_+ E. \]
When $E \to B$ is instead a principal $S^1$ bundle the same construction gives a dimension-shifting transfer
\[ \tr_f: \Sigma \Sigma^\infty_+ B \to \Sigma^\infty_+ E. \]
(cf. \cite[IV.2.13]{lms}) We drop the subscript $f$ when it is understood. As a special case, let $G$ be a topological group and $X$ an unbased $G$-space. If $G$ is discrete and finite then $X_{hH} \to X_{hG}$ is a covering map for every $H \leq G$, giving a transfer
\[ \tr^G_H : \Sigma^\infty_+ X_{hG} \ra \Sigma^\infty_+ X_{hH}. \]
When $G = S^1$ and $H = C_n$ is a proper subgroup, then $X_{hH} \to X_{hG}$ is an $S^1/C_n$-bundle, and so the circle transfer is a map
\[ \tr^{S^1}_{C_n}: \Sigma \Sigma^\infty_+ X_{hS^1} \ra \Sigma^\infty_+ X_{hC_n}. \]
Transfers in general are compatible with composition. For $C_{p^n} \to C_{p^{n+1}} \to S^1$, this means the circle transfer gives a map into the inverse limit
\begin{equation}\label{eq:inverse_lim_of_transfers}
\Sigma \Sigma^\infty_+ X_{hS^1} \ra \underset{n}\holim\, \Sigma^\infty_+ X_{hC_{p^n}}.
\end{equation}
We will use the following specialization of \cite[4.4.9]{madsen_survey}.
\begin{lem}\label{lem:inverse_lim_of_transfers}
The map \eqref{eq:inverse_lim_of_transfers} is an equivalence after completion at $p$.
\end{lem}

Next we recall some standard facts concerning the equivariant stable homotopy category \cite{lms}. We are interested in the case of $G$ a finite cyclic group, but we state the results for $G$ finite abelian. We recall that the homotopy category of spectra $\mc S$ has an analog $\mc S^G$ of genuinely $G$-equivariant spectra. A modern construction of $\mc S^G$ goes by inverting the $\pi_*$-isomorphisms in the category of orthogonal $G$-spectra \cite[III.3]{mandell2002equivariant}. For every inclusion of groups $H \leq G$ there is a forgetful change of groups functor $\mc S^G \to \mc S^H$, and for simplicity we omit this from the notation.

We say that an equivariant equivalence of $G$-spaces $X \to Y$ is a map inducing weak equivalences on the fixed points $X^H \to Y^H$ for all $H \leq G$. Every based $G$-cell complex $X$ has a suspension spectrum $\Sigma^\infty X$, and this construction sends equivariant equivalences of spaces to equivalences of $G$-spectra.

The category of $G$-spectra is equipped with two fixed-point functors, the \emph{genuine fixed points} and the \emph{geometric fixed points}
\[ (-)^H: \mc S^G \to \mc S^{G/H} \qquad \Phi^H(-): \mc S^G \to \mc S^{G/H}. \]
There are canonical isomorphisms $E^G \cong (E^H)^{G/H}$ for which we do not introduce additional notation.
Each of these functors measures weak equivalences, in the sense that a map of $G$-spectra $E \to E'$ is an equivalence iff it gives an equivalence on genuine fixed point spectra for all $H \leq G$, and similarly with the geometric fixed points. When $H \leq G$ there is a natural inclusion of fixed points map
\begin{equation}\label{eq:frobenius}
F: E^G \ra E^H
\end{equation}
for $G$-spectra $E$. In general there is no such map for the geometric fixed points, but on suspension spectra we have natural equivalences
\[ \Phi^H(\Sigma^\infty X) \simeq \Sigma^\infty X^H \]
and so one may simply include the fixed point spaces
\[ \Sigma^\infty \iota: \Sigma^\infty X^G \ra \Sigma^\infty X^H. \]
These two notions of fixed points are connected by a ``restriction'' natural transformation, giving for each $G$-spectrum $E$ a map of $G/H$ spectra
\begin{equation}\label{eq:restriction}
r^H: E^H \to \Phi^H E.
\end{equation}
When $E$ is a suspension spectrum $\Sigma^\infty_+ X$, the maps $r^H$ are split by natural inclusion maps \cite[II.3.14.(i)]{lms}
\[ s^H: \Sigma^\infty (X^H) \to (\Sigma^\infty X)^H. \]
A quick inspection in the category of orthogonal spectra gives
\begin{lem}\label{lem:frobenius_inclusion}
We have the commuting square in the homotopy category
\[ \xymatrix{
\Sigma^\infty (X^G) \ar[r]^-{s^G} \ar[d]^-{\Sigma^\infty \iota} & (\Sigma^\infty X)^G \ar[d]^-F \\
\Sigma^\infty (X^H) \ar[r]^-{s^H} & (\Sigma^\infty X)^H
} \]
\end{lem}
\noindent
The transfer $\tr^H$ has a variant which gives a map of $G/H$ spectra
\[ \ti\tr^H: \Sigma^\infty_+ X_{hH} \ra (\Sigma^\infty_+ X)^H. \]
The composition of $\ti\tr^H$ with $F: (\Sigma^\infty_+ X)^H \to \Sigma^\infty_+ X$ is precisely $\tr^H$. We also recall the following more general statement from \cite[(4.1.6)]{madsen_survey}.
\begin{lem}\label{lem:frobenius_transfer}
The following commutes in the homotopy category of spectra:
\[ \xymatrix{
\Sigma^\infty X_{hG} \ar[r]^-{\ti\tr^G} \ar[d]^-{\tr^G_H} & (\Sigma^\infty X)^G \ar[d]^-F \\
\Sigma^\infty X_{hH} \ar[r]^-{\ti\tr^H} & (\Sigma^\infty X)^H
} \]
\end{lem}
We end our exposition with the tom Dieck splitting, which we state for unbased $G$-spaces $X$. For a suspension spectrum $\Sigma^\infty_+ X$, the sum of the composites
\[ \xymatrix{ \Sigma^\infty_+ (X^H)_{hG/H} \ar[r]^-{s^H} & ((\Sigma^\infty_+ X)^H)_{hG/H} \ar[r]^-{\ti\tr^{G/H}} & ((\Sigma^\infty_+ X)^H)^{G/H} \ar[r]^-\cong & (\Sigma^\infty_+ X)^G } \]
gives an equivalence \cite[V.11.1]{lms}
\[ (\Sigma^\infty_+ X)^G \simeq \bigvee_{H \leq G} \Sigma^\infty_+ (X^H)_{hG/H}. \]
In addition, the restriction map $r^G$ agrees with the projection map of $(\Sigma^\infty_+ X)^G$ to the summand $\Sigma^\infty_+ X^G$ of the above splitting. The inclusion of fixed points has a more interesting effect.
\begin{prop}\label{prop:frobenius_on_splitting}
Let $H \leq G$ be an inclusion of abelian groups. Then the inclusion of fixed points
\[ \bigvee_{K \leq G} \Sigma^\infty_+ (X^K)_{hG/K} \simeq (\Sigma^\infty_+ X)^{G} \overset{F}\ra (\Sigma^\infty_+ X)^H \simeq \bigvee_{L \leq H} \Sigma^\infty_+ (X^L)_{hH/L} \]
takes the summand of $K$ to the summand of $L = H \cap K$ by the composite
\[ \xymatrix @C=4em{
\Sigma^\infty_+ (X^K)_{hG/K} \ar[r]^-{\tr^{G/K}_{H/(H \cap K)}} & \Sigma^\infty_+ (X^K)_{hH/(H \cap K)} \ar[r]^-{\Sigma^\infty \iota} & \Sigma^\infty_+ (X^{H \cap K})_{hH/(H \cap K)}  }\]
\end{prop}

\begin{proof}
Note that the transfer makes sense because the map $H/(H \cap K) \to G/K$ is injective. In the following diagram, the top-left square commutes by Lemma \ref{lem:frobenius_transfer}, the bottom-right by Lemma \ref{lem:frobenius_inclusion}, and the remaining two by naturality of $F$ and $\ti\tr$.
\[ \xymatrix @C=5em{
\Sigma^\infty_+ (X^K)_{hG/K} \ar[r]^-{\tr^{G/K}_{H/(H \cap K)}} \ar[d]^-{\ti\tr^{G/K}} & \Sigma^\infty_+ (X^K)_{hH/(H\cap K)} \ar[r]^-{(\Sigma^\infty \iota)_{hH/(H\cap K)}} \ar[d]^-{\ti\tr^{H/(H \cap K)}} & \Sigma^\infty_+ (X^{H \cap K})_{hH/(H\cap K)} \ar[d]^-{\ti\tr^{H/(H \cap K)}} \\
(\Sigma^\infty_+ X^K)^{G/K} \ar[r]^-F \ar[d]^-{s^K} & (\Sigma^\infty_+ X^K)^{H/(H\cap K)} \ar[r]^-{(\Sigma^\infty \iota)^{H/(H\cap K)}} \ar[d]^-{s^K} & (\Sigma^\infty_+ X^{H \cap K})^{H/(H\cap K)} \ar[d]^-{s^{(H \cap K)}} \\
(\Sigma^\infty_+ X)^G \ar[r]^-F & ((\Sigma^\infty_+ X)^K)^{H/(H \cap K)} \ar[r]^-F & (\Sigma^\infty_+ X)^H
} \]
Note that the bottom-center may be identified with $(\Sigma^\infty_+ X)^{HK}$.
\end{proof}
\noindent
As a special case, when $K \leq H \leq G$, the summand $\Sigma^\infty_+ (X^K)_{hG/K}$ is taken to $\Sigma^\infty_+ (X^K)_{hH/K}$ by the transfer $\tr^{G/K}_{H/K}$. When 
$H \leq K = G$, the summand $\Sigma^\infty_+ X^G$ is taken to $\Sigma^\infty_+ X^H$ by the inclusion $\Sigma^\infty \iota$.

\section{$K$-theory and topological cyclic homology of $DS^1$.}\label{sec:tc}

In this final section we calculate $TC(DS^1)$ after $p$-completion, proving Theorem \ref{thm:intro_tc_splitting} and Corollary \ref{cor:intro_coassembly_zero}. Throughout we assume the basic properties of $p$-completion of spectra \cite[\S 1-2]{bousfield_spectra}. We first recall the definition of $TC$. Of course, the primary use of $TC(R)$ is to approximate $K$-theory by the cyclotomic trace map $K(R) \to TC(R)$ \cite{bhm}. 
\begin{df}  \cite[\S 2.5]{madsen_survey}
If $R$ is a ring spectrum and $p$ is a prime, we regard $THH(R)$ as a genuine $C_{p^n}$ spectrum for all $n$ by forgetting along the homomorphisms $C_{p^n} \to S^1$. In this context the inclusion of fixed points map \eqref{eq:frobenius} is named the \emph{Frobenius} map
\[ F: \xymatrix{ THH(R)^{C_{p^n}} \ar[r] & THH(R)^{C_{p^{n-1}}} }. \]
The \emph{restriction} map $R$ is the composition of $r$ from \eqref{eq:restriction} and the cyclotomic structure map $c$ of $THH$ \eqref{eq:cyclotomic}: 
\[ R: \xymatrix @C=4em{ THH(R)^{C_{p^n}} \ar[r]^-{(r^{C_p})^{C_{p^{n-1}}}} & (\Phi^{C_p}THH(R))^{C_{p^{n-1}}} \ar[r]^-{(c_p)^{C_{p^{n-1}}}}_-\sim & THH(R)^{C_{p^{n-1}}} }. \]
The \emph{topological restriction homology} $TR(R;p)$ is the homotopy inverse limit of the fixed points $THH(R)^{C_{p^n}}$ under the maps $R$. As $F$ and $R$ commute, this limit inherits a self-map $(F - \id)$, whose homotopy fiber is the \emph{topological cyclic homology} $TC(R;p)$. Up to $p$-completion this is equivalent to $TC(R)$, defined as above but using all integers instead of just powers of $p$.
\end{df}

In the example of $R = DS^1$, $THH(R)$ is an equivariant suspension spectrum, up to a shift by a trivial representation (Corollary \ref{cor:intro_thh_splitting}):
\[ THH(D S^1) \simeq \Sph \vee \Sigma^{-1}\left( \bigvee_{n=1}^\infty \Sigma^\infty_+ S^1_{C_n} \right) \]
Here the $C_n$ subscript denotes orbits. Of course, $DS^1$ is an augmented ring spectrum, so the first summand $\Sph$ splits off in a way that respects $F$ and $R$, and we may safely ignore it. We let $\ti{THH}(DS^1)$ refer to the remaining summands. We rewrite this as
\[ \ti{THH}(D S^1) \simeq \Sigma^{-1}\left( \bigvee_{n=1}^\infty \Sigma^\infty_+ S^1_{C_n} \right) = \Sigma^{-1}\Sigma^\infty_+ X, \qquad X = \coprod_{n \geq 1} S^1_{C_n}. \]

By Remark \ref{rem:cm}, each map $c_m$ sends the $C_m$-fixed points of the $(mn)$th summand of $X$ to the $n$th summand by an equivalence, for all $n \geq 1$. These maps must agree up to $S^1$-equivariant homotopy with the obvious homeomorphism
\begin{equation}\label{eq:easy_cm}
\xymatrix{ (S^1_{C_{mn}})^{C_m} = S^1_{C_{mn}} \ar[r]^-\sim & S^1_{C_n} }
\end{equation}
(The equivariance takes the $S^1/C_m$-action on the left to the $S^1$-action on the right along a group isomorphism $S^1 \cong S^1/C_m$.) In fact, with a little more work one can check that $c_m$ agrees with \eqref{eq:easy_cm} on the nose.

As in \cite{bhm}, because $THH(DS^1)$ gives a (shift of an) equivariant suspension spectrum $\Sigma^\infty_+ X$, the restriction map is split. (In particular, the $r^{C_p}$ is split by $s^{C_p}$.) Along the tom Dieck splitting, this identifies the restriction map $R$ with a map that simply deletes the homotopy orbit spectrum $\Sigma^\infty_+ X_{hC_{p^n}}$:
\[ \xymatrix{
\Sigma^\infty_+ X_{hC_{p^n}} & \Sigma^\infty_+ (X^{C_p})_{hC_{p^{n-1}}} \ar[d]_-\cong^-{c_p} & \ldots & \Sigma^\infty_+ (X^{C_{p^{n-1}}})_{hC_p} \ar[d]_-\cong^-{c_p} & \Sigma^\infty_+ X^{C_{p^n}} \ar[d]_-\cong^-{c_p} \\
\mbox{} & \Sigma^\infty_+ X_{hC_{p^{n-1}}} & \ldots & \Sigma^\infty_+ (X^{C_{p^{n-2}}})_{hC_p} & \Sigma^\infty_+ X^{C_{p^{n-1}}}
} \]
On the other hand, by Proposition \ref{prop:frobenius_on_splitting} the action of $F$ is by transfer maps and a single inclusion of fixed points map:
\[ \xymatrix{
\Sigma^\infty_+ X_{hC_{p^n}} \ar[dr]_(.4){\tr^{C_{p^n}}_{C_{p^{n-1}}}} &
\ldots &
\Sigma^\infty_+ (X^{C_{p^{n-2}}})_{hC_{p^2}} \ar[dr]_(.4){\tr^{C_{p^2}}_{C_p}} &
\Sigma^\infty_+ (X^{C_{p^{n-1}}})_{hC_p} \ar[dr]_(.4){\tr^{C_p}} &
\Sigma^\infty_+ X^{C_{p^n}} \ar[d]^-{\Sigma^\infty \iota} \\
\mbox{} &
\Sigma^\infty_+ X_{hC_{p^{n-1}}} &
\ldots &
\Sigma^\infty_+ (X^{C_{p^{n-2}}})_{hC_p} &
\Sigma^\infty_+ X^{C_{p^{n-1}}}
} \]

\noindent
It follows that the inverse limit of the restriction maps is a product:
\[ \ti{TR}(DS^1;p) \simeq \Sigma^{-1} \prod_{j\geq 0} \Sigma^\infty_+ X_{hC_{p^j}} \]
Implicit in this is an identification of $X^{C_{p^n}}$ with $X$ along the maps $c_p$, for all $n$. The Frobenius map $F$ acts on this product by sending the $j$th factor to the $(j-1)$st factor by a transfer if $j \geq 1$. On the 0th factor, the action becomes
\[ F: \xymatrix{ \Sigma^\infty_+ X \ar@{<-}[r]^-{c_p}_-\cong & \Sigma^\infty_+ X^{C_p} \ar[r]^-{\Sigma^\infty \iota} & \Sigma^\infty_+ X } \]
which we abbreviate to $\Delta_p$. In the classical case where $X$ is a free loop space, $\Delta_p$ is a $p$-fold power map. In our case, $\Delta_p$ sends the $n$th summand of $X$ to the $(pn)$th summand by an equivalence.

Next we form the map of fiber sequences (cf. \cite[5.19]{bhm})
\begin{equation}\label{eq:fiber_seqs}
\xymatrix{
\Sigma^{-1} \Sigma^\infty_+ X \ar[r] \ar[d]^-{\Delta_p - \id} &
\Sigma^{-1} \prod_{j\geq 0} \Sigma^\infty_+ X_{hC_{p^j}} \ar[r] \ar[d]^-{F - \id} &
\Sigma^{-1} \prod_{j\geq 1} \Sigma^\infty_+ X_{hC_{p^j}} \ar[d]^-{F - \id} \\
\Sigma^{-1} \Sigma^\infty_+ X \ar[r] &
\Sigma^{-1} \prod_{j\geq 0} \Sigma^\infty_+ X_{hC_{p^j}} \ar[r] &
\Sigma^{-1} \prod_{j\geq 0} \Sigma^\infty_+ X_{hC_{p^j}} }
\end{equation}
The fiber of the middle column is our desired $\ti{TC}(DS^1;p)$.
The fiber of the right-hand column rearranges to the homotopy limit of $\Sigma^{-1} \Sigma^\infty_+ X_{hC_{p^j}}$ under the transfer maps $\Sigma^{-1}\tr^{C_{p^n}}_{C_{p^{n-1}}}$. By Lemma \ref{lem:inverse_lim_of_transfers} this agrees up to $p$-completion with $\Sigma^\infty_+ X_{hS^1}$. We observe for all $n \geq 1$ that
\[ (S^1_{C_n})_{hS^1} \simeq (BC_n \times S^1_{C_n})_{h(S^1/C_n)} \simeq BC_n, \]
and so the fiber of the right-hand column becomes the wedge of suspension spectra of $BC_n$ for all $n \geq 1$.

We next make explicit a lemma that was used implicitly in \cite[5.17]{bhm}, to assemble this together into a single homotopy pullback square.
\begin{lem}
Given a map of split fiber sequences of spectra
\[
\xymatrix{
A \ar[d]^-f \ar[r]^-{i_A} & A \times B \ar[d]^-{\phi} \ar[r]^-{\pi_B} & B \ar[d]^-g \\
A' \ar[r]^-{i_{A'}} & A' \times B' \ar[r]^-{\pi_{B'}} & B' }
\qquad 
\xymatrix @R=-0.5em { {} \\ {\phi = \begin{bmatrix} f & h \\ 0 & g \end{bmatrix}} } \]
The homotopy fiber of $\phi$ sits in a homotopy pullback square
\[ \xymatrix @C=5em {
F(\phi) \ar[d]^-{\pi_A \circ i(\phi)} \ar[r]^-{F(\pi_B,\pi_{B'})} & F(g) \ar[d]^-{-h \circ i(g)} \\
A \ar[r]^-f & A' } \]
where $i(f)$ denotes the inclusion of the homotopy fiber.
\end{lem}

\begin{proof}
Take a shift and loop space of all spectra in the diagram. Then the negation $-h$ may be interpreted concretely as reversal of loops, and $\phi$ as taking a pair of loops $(\alpha,\beta)$ to $(f(\alpha) \cdot h(\beta), g(\beta))$, where $\cdot$ denotes concatenation. Under these conventions, on each spectrum level, both $F(\phi)$ and the homotopy pullback of the given square are identified up to homeomorphism with the space of choices of $\alpha \in \Omega A$, $\beta \in \Omega B$, a nullhomotopy of $g(\beta)$, and a homotopy $f(\alpha) \sim h(\beta)$. This identification may be done in a manner that is functorial in $A$ and $B$ and hence commutes with the structure maps of the two spectra.
\end{proof}

Applying this lemma to \eqref{eq:fiber_seqs} gives a homotopy pullback square after $p$-completion
\[ \xymatrix @C=5em @R=2em{
\ti{TC}(D S^1) \ar[r] \ar[d] & \bigvee_{n=1}^\infty \Sigma^\infty_+ BC_n \ar[d]^-{\bigvee \tr^{S^1}} \\
\Sigma^{-1}\bigvee_{n=1}^\infty \Sigma^\infty_+ S^1_{C_n} \ar[r]^-{\Delta_p - \id} & \Sigma^{-1}\bigvee_{n=1}^\infty \Sigma^\infty_+ S^1_{C_n}
} \]
We adopt the convention that $p^k$ is the highest power of $p$ dividing $n$. As the covering maps $BC_{p^k} \ra BC_n$ and $S^1_{C_{p^k}} \ra S^1_{C_n}$ are equivalences after $p$-completion, the square simplifies to
\[ \xymatrix @C=5em @R=2em{
\ti{TC}(D S^1) \ar[r] \ar[d] & \bigvee_{n=1}^\infty \Sigma^\infty_+ BC_{p^k} \ar[d]^-{\bigvee \tr^{S^1}} \\
\Sigma^{-1}\bigvee_{n=1}^\infty \Sigma^\infty_+ S^1_{C_{p^k}} \ar[r]^-{\Delta_p - \id} & \Sigma^{-1}\bigvee_{n=1}^\infty \Sigma^\infty_+ S^1_{C_{p^k}}
} \]
This splits into an infinite wedge of squares. There is one square for each equivalence class of positive integers, where $n \sim m$ if $\frac{m}{n}$ is a power of $p$.
Each equivalence class gives the same pullback square, and we denote the homotopy pullback by $E$:
\[ \xymatrix{
E \ar[r] \ar[d] & \bigvee_{k=0}^\infty \Sigma^\infty_+ BC_{p^k} \ar[d]^-{\bigvee \tr^{S^1}} \\
\Sigma^{-1}\bigvee_{k=0}^\infty \Sigma^\infty_+ S^1_{C_{p^k}} \ar[r]^-{\Delta_p - \id} & \Sigma^{-1}\bigvee_{k=0}^\infty \Sigma^\infty_+ S^1_{C_{p^k}}
} \]
In this final square, the cofiber of the bottom row is the colimit of the system of equivalences
\[ \xymatrix{ \Sigma^{-1} \Sigma^\infty_+ S^1 \ar[r]^-{\Delta_p}_-\sim & \Sigma^{-1} \Sigma^\infty_+ S^1_{C_p} \ar[r]^-{\Delta_p}_-\sim & \Sigma^{-1} \Sigma^\infty_+ S^1_{C_{p^2}} \ar[r]^-{\Delta_p}_-\sim & \ldots } \]
This is a single copy of $\Sigma^{-1} \Sigma^\infty_+ S^1$. We observe that $\bigvee_{k=0}^\infty \Sigma^\infty_+ BC_{p^k}$ maps into this by a wedge of circle transfers $\tr^{S^1}$, and we conclude that $E$ is the fiber of this map.

We observed that $\ti{TC}(DS^1)$ was a summand of $TC(DS^1)$ because $DS^1$ is an augmented ring. Of course, the complementary piece is $TC(\Sph)$, which is known to agree up to $p$-completion with $\Sph \vee \Sigma \mathbb{CP}^\infty_{-1}$ \cite[0.1]{bergsaker2010homology}. This finishes the proof of Theorem \ref{thm:intro_tc_splitting}:
\[ TC(D S^1) \overset{\mbox{}^\wedge_p}\simeq \Sph \vee \Sigma \mathbb{CP}^\infty_{-1} \vee \bigvee_{n \in \mathbb{N}} E \]

We end with a calculation that leads to Corollary \ref{cor:intro_coassembly_zero}. The calculation proceeds as follows. First, our splitting of $TC(DS^1)^\wedge_p$ gives a model before $p$-completion whose integral homology groups $H_n(-;\Z)$ are straightforward to calculate; we summarize the results in the table below.
\begin{table}[h]\label{table:homology}
\[ \begin{array}{c|llllllll}
\textup{spectrum} & H_{-2} & H_{-1} & H_0 & H_1 & H_2 & H_3 & H_4 & \ldots \\\hline
\Sph & 0 & 0 & \Z & 0 & 0 & 0 & 0 & \ldots \\
\Sigma\mathbb{CP}^\infty_{-1} & 0 & \Z & 0 & \Z & 0 & \Z & 0 & \ldots \\
E & \Z & 0 & \bigoplus_{k \geq 0} \Z & \bigoplus_{k \geq 0} \Z/p^k & 0 & \bigoplus_{k \geq 0} \Z/p^k & 0 & \ldots
\end{array} \]
\caption{Integral homology groups of the components of $TC(DS^1)^\wedge_p$.} \end{table}\vspace{-1em}
Along the Hurewicz map, these are isomorphic to the homotopy groups $\pi_n(-;\Z)$, so long as we work modulo torsion prime to $p$, and only in the range $n \leq 2p - 6$ \cite[4.1]{arlettaz1996hurewicz}. We then feed these results into the splittable short exact sequence \cite[2.5]{bousfield_spectra}
\[ 0 \ra \Ext(\Z/p^\infty,\pi_n(X)) \ra \pi_n(X^\wedge_p) \ra \Hom(\Z/p^\infty,\pi_{n-1}(X)) \ra 0 \]
and arrive at the homotopy groups for $TC(DS^1)^\wedge_p$ in this range.

We give the rationalizations of these groups in the table below. They repeat with period 4 starting at $\pi_2^\Q$, but it is important to remember that this method only gives the correct answer for $\pi_n^\Q$ when $n \leq 2p - 6$, or equivalently $p \geq \frac{n}2 + 3$.

\begin{table}[h]
\[ \begin{array}{c|ccccccccc}
\textup{spectrum} & \pi_{-2}^\Q & \pi_{-1}^\Q & \pi_0^\Q & \pi_1^\Q & \pi_2^\Q & \pi_3^\Q & \pi_4^\Q & \pi_5^\Q & \pi_6^\Q \\\hline
K(\Sph) & 0 & 0 & \Q & 0 & 0 & 0 & 0 & \Q & 0 \\
D S^1 \sma K(\Sph) & 0 & \Q & \Q & 0 & 0 & 0 & \Q & \Q & 0 \\
TC(\Sph)^\wedge_p & 0 & \Q_p & \Q_p & \Q_p & 0 & \Q_p & 0 & \Q_p & 0 \\
(D S^1 \sma TC(\Sph))^\wedge_p & \Q_p & \Q_p^2 & \Q_p^2 & \Q_p & \Q_p & \Q_p & \Q_p & \Q_p & \Q_p \\
E^\wedge_p & \Q_p & 0 & A & B & 0 & B & 0 & B & 0 \\
TC(D S^1)^\wedge_p & A & \Q_p & A & B_\infty & 0 & B_\infty & 0 & B_\infty & 0 \\
\end{array} \]
\caption{Rational homotopy groups of $TC(D S^1)^\wedge_p$ and related spectra.} \end{table}

\noindent
The vector spaces $A$, $B$, and $B_\infty$ are quite large. They are defined by
\[ \begin{array}{l}
A = \Ext(\Z/(p^\infty),\overset\infty\bigoplus\, \Z) \otimes \Q \\
B = \Ext(\Z/(p^\infty),\oplus_{k=0}^\infty\, \Z/p^k) \otimes \Q \\
B_\infty = \Ext(\Z/(p^\infty),\overset\infty\bigoplus \oplus_{k=0}^\infty\, \Z/p^k) \otimes \Q
\end{array} \]
and each one contains $\Q^n_p = \Ext(\Z/(p^\infty),\overset{n}\bigoplus\, \Z) \otimes \Q$ as a retract, for every $n \geq 0$.

Now we may draw conclusions about the rational behavior of $K(DS^1)$ under the coassembly map. We recall that the coassembly map is a natural transformation of the form
\[ c\alpha: \Phi(X) \ra F(X_+,\Phi(*)) \simeq DX \sma \Phi(*) \]
where $\Phi$ is any contravariant, homotopy-invariant functor from unbased finite complexes $X$ to spectra, with $\Phi(\emptyset) \simeq *$ (\cite[\S 5]{cohen2009umkehr}, cf. \cite{weiss1993assembly}). The two functors $K(D(-))$ and $TC(D(-))^\wedge_p$ both satisfy these hypotheses, and they are connected by the cyclotomic trace $\textup{trc}: K(DX) \to TC(DX)^\wedge_p$. By the formal properties of coassembly, this gives a commuting square:
\begin{equation}\label{eq:coassembly_square}
\xymatrix{
K(DS^1) \ar[r]^-{c\alpha} \ar[d]^-{\textup{trc}} & DS^1 \sma K(\Sph) \ar[d]^-{\id \sma \textup{trc}} \\
TC(DS^1)^\wedge_p \ar[r]^-{c\alpha} & DS^1 \sma TC(\Sph)^\wedge_p }
\end{equation}
If $p$ is a regular prime \cite[Ch.1]{washington1997introduction}, then the trace map $K(\Sph) \to TC(\Sph)^\wedge_p$ on the right-hand side of \eqref{eq:coassembly_square} is rationally injective \cite[proof of 4.5.4]{madsen_survey}.

Suppose $i \geq 1$ and we choose a regular prime $p \geq 2i + 3$. Then the pattern in the above table extends to $\pi_{4i}^\Q$, and \eqref{eq:coassembly_square} becomes on $\pi_{4i}^\Q$ the square of abelian groups
\begin{equation*}
\xymatrix @R=1.5em @C=3em{
? \ar[r]^-{c\alpha} \ar[d] & \Q \ar[d]^-{\neq 0} \\
0 \ar[r]^-{c\alpha} & \Q^\wedge_p }
\end{equation*}
We conclude that the top horizontal map is zero, proving Corollary \ref{cor:intro_coassembly_zero}.

\bibliographystyle{amsalpha}
\bibliography{thhdx}{}

Department of Mathematics \\
University of Illinois at Urbana-Champaign \\
1409 W Green St \\
Urbana, IL 61801 \\
\texttt{cmalkiew@illinois.edu}

\end{document}